\documentclass[12pt]{amsart}
\usepackage{bbm}
\usepackage{stmaryrd}
\usepackage{amsrefs}
\usepackage{amssymb}
\usepackage{amsmath}
\usepackage{proof}
\usepackage{graphicx}
\usepackage{mathrsfs}
\usepackage{tikz-cd}
\usepackage{centernot}
\usepackage{shuffle}
\usepackage{enumerate}
\usepackage{relsize}
\usepackage{scalerel}
\usepackage{bm}
\usepackage{amsaddr}

\usepackage{fullpage}

\newtheorem{theorem}{Theorem}[section]
\newtheorem{lemma}[theorem]{Lemma}
\newtheorem{proposition}[theorem]{Proposition}
\newtheorem{corollary}[theorem]{Corollary}

\theoremstyle{definition}
\newtheorem{definition}[theorem]{Definition}

\theoremstyle{remark}

\theoremstyle{plain}
\newtheorem*{theorem*}{Theorem}
\newtheorem*{lemma*}{Lemma}
\newtheorem*{proposition*}{Proposition}

\theoremstyle{definition}
\newtheorem*{definition*}{Definition}

\theoremstyle{remark}

\renewcommand{\:}{\colon}
\newcommand{\subsetof}{\subseteq}

\newcommand{\To}{\rightarrow}
\newcommand{\tensor}{\otimes}
\newcommand{\iso}{\cong}

\newcommand{\Hom}{\mathrm{Hom}}

\newcommand{\suchthat}{\,|\,}

\newcommand{\CC}{\mathbb C}

\newcommand{\RR}{\mathbb R}

\newcommand{\ZZ}{\mathbb Z}

\newcommand{\eval}{\mathrm{eval}}
\newcommand{\RP}{\mathbb{R}\mathrm{P}}
\newcommand{\CP}{\mathbb{C}\mathrm{P}}
\newcommand{\cst}{\mathrm{cnst}}

\begin{document}

\title{Quantum extensions of ordinary maps}
\author{\footnotesize Andre Kornell}
\address{ \footnotesize Department of Mathematics\\ University of California, Davis}
\email{kornell@math.ucdavis.edu}

\begin{abstract}
We define a loop to be quantum nullhomotopic if and only if it admits a nonempty quantum set of extensions to the unit disk. We show that the canonical loop in the unit circle is not quantum nullhomotopic, but that every loop in the real projective plane is quantum nullhomotopic. Furthermore, we apply Kuiper's theorem to show that the canonical loop admits a continuous family of extensions to the unit disk that is indexed by an infinite quantum space. We obtain these results using a purely topological condition that we show to be equivalent to the existence of a quantum family of extensions of a given map.
\end{abstract}

\maketitle

\noindent\textbf{Quantum families of extensions.} Noncommutative mathematics draws an analogy between unital C*-algebras and compact Hausdorff spaces. Commutative unital C*-algebras correspond bijectively to compact Hausdorff spaces, up to isomorphism of the former and up to homeomorphism of the latter, so noncommutative unital C*-algebras are viewed as ``quantum compact Hausdorff spaces''. Specifically, each compact Hausdorff space $X$ corresponds to the commutative C*-algebra $C(X)$ of continuous complex-valued functions on $X$, and conversely, each noncommutative unital C*-algebra is imagined to have the same form, but for some fictitious space $X$, whose local topology is quantum in the tautological sense that some real-valued functions on this space fail to commute, in the manner of incompatible observables.

This correspondence between compact Hausdorff spaces and commutative unital C*-algebras extends to a correspondence between continuous maps and unital $*$-homomorphisms. Each map $f$ from a compact Hausdorff space $X$ to a compact Hausdorff space $Y$ induces a unital $*$-homomorphism from $C(Y)$ to $C(X)$, by precompostion. We thus have a contravariant equivalence of categories. Each unital $*$-homomorphism between noncommutative unital C*-algebras is thus viewed as a continuous function between the corresponding quantum spaces, but in the opposite direction. This ``noncommutative metaphor'' can be extended further and further \cite{GarciaBondiaVarillyFigueroa}. For the purposes of the this paper, we only recall that $C(X \sqcup Y) \iso C(X) \oplus C(Y)$ and $C(X \times Y) \iso C(X) \otimes C(Y)$, for compact Hausdorff spaces $X$ and $Y$, so we take the direct sum and the minimal tensor product of C*-algebras to be the appropriate generalizations of disjoint union and Cartesian product of compact Hausdorff spaces. The minimal tensor product is preferred over other notions because it distributes over the direct sum.

Let $S$ and $T$ be compact Hausdorff spaces, and $f$ an ordinary map from $S$ to $T$. If $S$ is a subspace of a larger compact Hausdorff space $\tilde S$, then an extension of $f$ to $\tilde S$ is of course a map $\tilde f\: \Tilde S \To T$ making the following diagram commute:
$$
\begin{tikzcd}
\tilde S \arrow[dotted]{rrd}{\tilde f}
&
&
\\
S \arrow[hook]{u}{j} \arrow{rr}[swap]{f}
&
&
T
\end{tikzcd}
$$
We write $j\: S \hookrightarrow \tilde S$ for the inclusion map. A \textit{family} of extensions indexed by a compact Hausdorff space $Y$ is instead a continuous function $\tilde f$ of two variables making making the following diagram commute:
$$
\begin{tikzcd}
\tilde S \arrow[dotted]{rrd}{\tilde f} \times Y
&
&
\\
S \times Y \arrow[hook]{u}{j \times \mathrm{id}} \arrow{r}[swap]{f \times \mathrm{id}}
&
T \times Y \arrow{r}[swap]{\mathrm{proj}_1}
&
T
\end{tikzcd}
$$

The existence of a nonempty family of extensions indexed by an ordinary compact Hausdorff space is of course equivalent to the existence of a single extension. We will show that there is no such equivalence when we consider nonempty families of extensions indexed by a \textit{quantum} compact Hausdorff space. To make this claim precise, we first apply the functor $C$ to the diagram above, to obtain a diagram in the category of unital C*-algebras and unital $*$-homomorphisms:
$$
\begin{tikzcd}
C(\tilde S) \otimes C(Y) \arrow{d}[swap]{j^\star \tensor \mathrm{id}}
&
&
\\
C(S) \otimes C(Y)
&
C(T) \otimes C(Y) \arrow{l}{f^\star \tensor \mathrm{id}}
&
C(T) \arrow[dotted]{llu} \arrow{l}{\mathrm{proj}_1^\star}
\end{tikzcd}
$$
Allowing the index space $Y$ to be quantum, we replace $C(Y)$, with an arbitrary unital C*-algebra $B$. We say that the quantum space is nonempty just in case the C*-algebra $B$ is nonzero. Thus, we ask whether there exists a nonzero unital C*-algebra $B$, and a unital normal $*$-homomorphism $\phi\: C(T) \To C(\tilde S) \tensor B$ making the following diagram commute:
$$
\begin{tikzcd}
C(\tilde S) \otimes B \arrow{d}[swap]{j^\star \tensor \mathrm{id}}
&
&
\\
C(S) \otimes B
&
C(T) \otimes B \arrow{l}{f^\star \tensor \mathrm{id}}
&
C(T) \arrow[dotted]{llu}[swap]{\phi} \arrow{l}{a \tensor 1 \mapsfrom a}
\end{tikzcd}
$$

The core result of the paper provides a purely topological condition equivalent to the existence of a family of extensions indexed by a quantum compact Hausdorff space.

\setcounter{section}{1}

\setcounter{theorem}{5}

\begin{theorem}
Let $B$ be a unital C*-algebra. Let $\imath \: T \To \Hom(C(T), B)$ be the map taking each point of $T$ to evaluation at that point. There exists a unital $*$-homomorphism $C(T) \To C(\tilde S) \tensor B$ making the above diagram commute if and only if the ordinary map $\imath \circ f$ extends to $\tilde S$:

\
\vspace{-.3in}
$$
\begin{tikzcd}
\tilde S \arrow[dotted]{rrd}{\tilde f}
&
&
\\
S \arrow[hook]{u}{j}\arrow{r}[swap]{f}
&
T  \arrow{r}[swap]{\imath}
&
\mathrm{Hom}(C(T), B )
\end{tikzcd}
$$
\end{theorem}
\noindent Thus, the functor $\mathrm{Hom}(C(\,- \,), B)$ is analogous to the quantum monad of Abramsky, Barbosa, de Silva, and Zapata \cite{AbramskyBarbosaDeSilvaZapata}.

\smallskip

\noindent\textbf{Quantum nullhomotopic loops.} Recall that a loop $f\: S^1 \To T$ is nullhomotopic if and only if $f$ extends to a map on the unit disk $D^2$. As an application of the above theorem, we investigate the corresponding quantum notion. Setting $S = S^1$ and $\tilde S = D^2$, we ask whether a map $f\: S^1 \To T$ admits a family of extensions indexed by a nonempty quantum compact Hausdorff space.

This notion of quantum nullhomotopy degenerates if we allow arbitrary quantum compact Hausdorff spaces as index spaces. As a consequence of Kuiper's theorem \cite{Kuiper} we show that the canonical loop $S^1 \To S^1$ admits a quantum family of extensions to the unit disk:

\setcounter{section}{2}

\setcounter{theorem}{5}

\begin{theorem}\label{canonical}
Let $H$ be an infinite-dimensional Hilbert space, and let $L(H)$ be the unital C*-algebra of bounded operators on $H$. There is a unital $*$-homomorphism $\phi$ such that the following diagram commutes:
$$
\begin{tikzcd}
C(D^2) \tensor L(H)
\arrow{d}[swap]{j^\star \tensor \mathrm{id}}
&
&
\\
C(S^1) \tensor L(H)
&
&
C(S^1)
\arrow{ll}{ a \tensor 1 \mapsfrom a}
\arrow{ull}[swap]{\phi}
\end{tikzcd}
$$
\end{theorem}
\noindent It follows that every loop in every compact Hausdorff space admits a quantum family of extensions to the unit disk indexed by the quantum compact Hausdorff space corresponding to $L(H)$. Therefore, we make the following definition:

\setcounter{theorem}{0}

\begin{definition}
A loop $f\: S^1 \To T$ is \underline{quantum nullhomotopic} just in case there is a nonzero finite-dimensional Hilbert space $H$ and unital $*$-homomorphism $\phi$ such that the following diagram commutes.

$$
\begin{tikzcd}
C(D^2) \otimes L(H) \arrow{d}[swap]{j^\star \tensor \mathrm{id}}
&
&
\\
C(S^1) \otimes L(H)
&
C(T) \otimes L(H) \arrow{l}{f^\star \tensor \mathrm{id}}
&
C(T) \arrow[dotted]{llu}[swap]{\phi} \arrow{l}{a \tensor 1 \mapsfrom a}
\end{tikzcd}
$$
\end{definition}
\noindent Equivalently, we may replace the arbitrary nonzero matrix algebra $L(H)$ with an arbitrary nonzero finite-dimensional C*-alegebra $B$ in this definition.

This notion of quantum nullhomotopy does not degenerate:
\setcounter{theorem}{3}
\begin{proposition}
The identity function $S^1 \To S^1$ is not quantum nullhomotopic as a loop in $S^1$.
\end{proposition}
\noindent Furthermore, it is distinct from the ordinary notion of nullhomotopy of loops:
\setcounter{section}{3}
\setcounter{theorem}{2}
\begin{corollary}
Every loop in $\RP^2$ is quantum nullhomotopic.
\end{corollary}
\setcounter{theorem}{4}
\begin{corollary}
There is a loop in $S^1 \vee S^1$ that is quantum nullhomotopic, but not nullhomotopic in the ordinary sense.
\end{corollary}

\smallskip
\newpage

\noindent \textbf{Quantum sets.}
If $X$ is a locally compact Hausdorff space, then $$C_0(X) = \{f\in (X, \CC) \suchthat \lim_{x \To \infty} f(x) = 0\}$$ is a C*-algebra, which is unital iff $X$ is compact. Thus, C*-algebras are commonly viewed to be the quantum generalization of locally compact Hausdorff spaces. In particular, quantum sets should correspond to a class of C*-algebras.

The noncommutative dictionary does not yet include a widely accepted quantum generalization of sets. If the C*-algebra of compact operators is taken to correspond to a quantum set, then in fact there does exist a nonempty quantum set of extensions of the canonical loop to the unit disk. The unital $*$-homomorphism $\phi$ of theorem \ref{canonical} can be viewed as a morphism, in the sense of Woronowicz \cite{GarciaBondiaVarillyFigueroa}, from $C(S^1)$ to $C(D^2) \tensor L_0(H)$, and thus, as a quantum family of extensions that is indexed by the quantum set corresponding the compact operator algebra $L_0(H)$.

The notion of quantum set preferred by the author excludes those quantum locally compact Hausdorff spaces that correspond to infinite-dimensional compact operator C*-algebras. Instead, quantum sets are taken to be those quantum locally compact Hausdorff spaces that correspond to $c_0$-direct sums of finite matrix C*-algebras. This notion of discreteness first appeared in the context of quantum Gelfand duality \cite{PodlesWoronowicz90} \cite{EffrosRuan94} \cite{VanDaele96} \cite{DeCommerKasprzakSkalskiSoltan16}. It was later promoted by the author \cite{Kornell18}.

Decomposing the C*-algebra of a quantum set into matrix algebras, we find that a loop is quantum nullhomotopic if and only if it admits a quantum set of extensions to the unit disk, that is, a quantum family of extensions indexed by a discrete quantum space.

\smallskip

\noindent \textbf{Quantum pseudotelepathy.} Quantum families of functions can be interpreted as quantum strategies for games in which two players, traditionally named Alice and Bob, cooperate against a Referee without communicating with one another. In some instances, Alice and Bob have a winning strategy that utilizes quantum entanglement, despite having no winning strategy classically. The availability of a quantum strategy is equivalent to the existence of a quantum set of functions of one or another kind \cite{CameronMontanaroNewmanSeveriniWinter}*{proposition 1} \cite{MancinskaRoberson} \cite{AbramskyBarbosaDeSilvaZapata} \cite{MustoReutterVerdon} \cite{Kornell18}*{1.2}.

In this context, the domain and codomain spaces are taken to be finite. However, it is also possible to interpret the quantum extension problem considered in this paper as an idealized game of the same kind. In the example of theorem \ref{canonical}, the Referee sends Alice and Bob elements of $D^2$, and they respond with elements of $S^1$. Alice and Bob lose if their responses disagree, provided that the Referee sent them both the same element of $D^2$, or if either Alice or Bob fails to mirror the Referee's move, provided that the Referee sent them an element of $S^1$.

I speculate that this idealized game is not so far removed from the physical world as one might imagine. The players could perhaps exchange continuous values as the momenta of particles, and the Referee could be certain of sending the same value to both Alice and Bob by using a mechanism such as pair creation. The strategies implemented by Alice and Bob would necessarily be continuous, as a physical limitation.

Preliminary computations suggest that it is possible to formulate some of the results of the present paper in terms of discrete games. However, these games retain some infinitary element, e. g., the number of possible moves is infinite, or the length of play is unbounded. Furthermore, the quantum advantage becomes probabilistic, rather than deterministic. Thus, the significance of such reformulations is uncertain.

\noindent \textbf{Acknowledgements.} I thank Neil Ross and Peter Selinger for organizing the Quantum Physics and Logic conference in Halifax, an enriching experience that led me to consider this quantum extension problem. I thank Greg Kuperberg for advising me to emphasize the topological aspect of the problem, in favor of the quantum communication aspect. I thank Rui Soares Barbosa and Sam Staton for giving me the opportunity to present this research at the Oxford Advanced Seminar on Informatic Structures.

\vspace{-.1in}

\setcounter{section}{0}

\section{quantum extensions}

Let $S$ and $T$ be compact Hausdorff spaces, and let $H$ be a nonzero finite-dimensional Hilbert space. In this section, we demonstrate a one-to-one correspondence between the unital $*$-homomorphisms from $C(T)$ to $C(S) \tensor L(H)$, and the continuous functions from $S$ to $\mathrm{Hom}(C(T), L(H))$. There are several reasonable ways to gloss the term $C(S) \tensor L(H)$ that turn out to be equivalent, and likewise for $\mathrm{Hom}(C(T), L(H))$. The C*-algebras $C(S)$ and $L(H)$ are both nuclear, so there is a unique cross-norm on their algebraic tensor product. In fact, their algebraic tensor product is already isomorphic to the C*-algebra of $\mathrm{dim}(H) \times \mathrm{dim}(H)$ matrices over $C(S)$, so the term $C(S) \tensor L(H)$ can refer equivalently to the maximum tensor product, the minimum tensor product, or the algebraic tensor product of the two C*-algebras.

The space $\mathrm{Hom}(C(T), L(H))$ consists of unital $*$-homomorphisms from $C(T)$ to $L(H)$. This set is typically equipped with the point-norm topology, which is characterized by the condition that a net $(\phi_\alpha)$ converges to $\phi_\infty$ if and only if for each function $b$ in $C(T)$, the net $(\phi_\alpha(b))$ converges to $\phi_\infty(b)$ in the operator norm topology. We instead equip $\mathrm{Hom}(C(T), L(H))$ with the kaonization of the compact-open topology, which we will later show to be equivalent to the point-norm topology. We do so to reason in Steenrod's convenient category of compactly generated spaces \cite{Steenrod}. We summarize a few of its basic aspects.

Steenrod's convenient category is a subcategory of the category of topological spaces and continuous functions that includes all locally compact spaces and all metrizable spaces. It is convenient primarily in the sense that it is Cartesian closed: for all objects $X$, $Y$, and $Z$ of the category, the space of continuous functions $C(Y,Z)$ is itself an object of the category, and there is a natural one-to-one correspondence between the continuous function from $X \times Y$ to $Z$ and the continuous functions from $X$ to $C(Y, Z)$. However, the category is inconvenient in the sense that the topology of the category-theoretic product $X \times Y$ is sometimes different from the usual product topology.

The objects of Steenrod's category are the so-called compactly generated Hausdorff spaces. These spaces are characterized by the criterion that any set whose intersection with every compact set is closed must itself be closed. Kaonization is the functor from Hausdorff spaces to compactly generated Hausdorff spaces that adds closed sets according to this criterion; it is typically denoted by the lower-case letter $k$. The category-theoretic product of two compactly generated Hausdorff spaces $X$ and $Y$ is the kaonization of the usual topological product. The morphisms of Steenrod's category are ordinary continuous functions, and the set of continuous functions $C(Y,Z)$ is canonically equipped with the kaonization of the compact-open topology on this set. If $X$ is locally compact, then the product topology on $X \times Y$ is already compactly generated. If $Y$ is compact, and $Z$ is metrizable, then the compact-open topology on $C(Y, Z)$ is also already compactly generated; it is simply the uniform topology. In particular, in Steenrod's category, $C(S) = C(S, \CC)$ and $C(T) = C(T, \CC)$ are canonically equipped with their usual norm topologies.

With this overview of the compactly generated spaces, we are ready to formally state and prove the claimed one-to-one correspondence.

\begin{definition}\label{k1}
Let $A$ and $B$ be unital C*-algebras. We write $\mathrm{Hom}(A, B)$ for the set of all unital $*$-homomorphisms from $A$ to $B$, equipped with the kaonization of the compact-open topology.
\end{definition}

The space $\mathrm{Hom}(A, B)$ is a subspace of the space $C(A, B)$, itself equipped with the kaonization of its compact open topology. It easy to see that for closed subsets, the kaonization of the subspace topology coincides with the subpace topology of the kaonization.

\begin{lemma}\label{k2}
We have a homeomorphism
$\mathrm{Hom} (A , C(S) \tensor B ) \; \iso \; C( S, \mathrm{Hom}(A , B )),$ natural in the compact Hausdorff space $S$, and the unital C*-algebras $A$ and $B$.
\end{lemma}

Implicitly, the morphisms of compact Hausdorff spaces are continuous functions, and the morphisms of unital C*-algebras are unital $*$-homomorphisms.

\begin{proof}
The compact-open topology on $C(S,B )$ is the topology of uniform convergence. It is a metrizable topology, so it is compactly generated; thus, the topology on $C(S,B )$ in Steenrod's convenient category is just the usual one. There is a well known isomorphism between the C*-algebras $C(S) \tensor B $ and $C(S, B )$ \cite{Murphy}*{theorem 6.4.17}, which is natural in $S$ and $B$, so it is sufficient to establish that $\mathrm{Hom} (A , C(S, B )) \iso C( S, \mathrm{Hom}(A , B )).$

We obtain this natural homeomorphism as a restriction of the following composition:
$$ C(A , C(S, B )) \; \iso \; C(A  \times S, B ) \;\iso\; C (S, C(A , B ))$$
Both natural homeomorphisms are instances of the Cartesian closedness of Steenrod's convenient category. Each continuous function $\pi$ from $A $ to $C(S, B )$ corresponds to a continuous function $\pi'$ from $S$ to $C(A ,B )$, which is defined by $\pi'(s)(a) = \pi(a)(s)$ for all $a \in A $, for all $s \in S$. Thus, if $\pi$ is a unital $*$-homomorphism, then for each $s \in S$, $\pi'(s)$ is a unital $*$-homomorphism, and vica versa, because the algebraic structure of $C(S, B )$ is defined pointwise. We conclude that the natural homeomorphism $ C(A , C(S, B )) \iso C (S, C(A , B ))$ restricts to a natural homeomorphism $\mathrm{Hom} (A , C(S, B )) \iso C( S, \mathrm{Hom}(A , B )).$
\end{proof}

For reference, we record that each unital $*$-homomorphism $\phi\: A  \To C(S) \tensor B $ corresponds to the continuous function $s \mapsto (\eval_s \tensor \mathrm{id}_B ) (\phi( \, \cdot\,))$.

\begin{lemma}\label{k3}
Let $V$ and $W$ be Banach spaces. Write $L(V,W)_1$ for the set of linear operators from $V$ to $W$ of norm at most $1$. The point-norm topology on $L(V,W)_1$ is equal to the compact-open topology.
\end{lemma}

The point-norm topology is generated by subbasis elements of the form $\{t \suchthat  \| t( v) -w \| < \epsilon \}$ for $v \in V$, $w \in W$, and $\epsilon >0$. The compact-open topology is of course generated by subbasis elements of the form $\{ t \suchthat t(K) \subsetof U\}$ for $K \subsetof V$ compact, and $U \subsetof W$ open. Both Banach spaces are equipped with their norm topologies.

\begin{proof}
It is well known that convergence in the compact-open topology is equivalent to uniform convergence on compact subsets of the domain, whenever the codomain is a metric space. The point-norm topology is just the topology of pointwise convergence. Convergence of the former kind clearly implies convergence of the latter kind, so it remains only to show the converse.

Let $(t_\alpha)$ be a net in $L(V, W)_1$ converging pointwise to a linear operator $t_\infty$, itself automatically of norm at most $1$. Let $K$ be a compact subset of $V$, and let $\epsilon >0$. Being compact, the set $V$ is totally bounded, so we can cover $K$ by a finitely many open balls $B_\epsilon(v_1), B_\epsilon(v_2), \ldots, B_\epsilon(v_n)$ of radius $\epsilon$, with centers $v_1, v_2, \ldots, v_n \in V$. Consider $\alpha$ sufficiently large so that $\|t_\alpha(v_1) - t(v_1)\| < \epsilon$, $\|t_\alpha(v_2) - t_\alpha(v_2)\| < \epsilon$, $\ldots$, and $\|t_\alpha(v_n) - t_\alpha(v_n)\| < \epsilon$. Every element $v$ of $K$ is within distance $\epsilon$ of some $v_i$, so we calculate:
\begin{align*}
\|t_\alpha(v) - t(v)\|  & \leq \|t_\alpha(v) - t_\alpha(v_i)\| + \|t_\alpha(v_i) - t_\infty(v_i)\| + \| t_\infty(v_i) - t_\infty(v)\|
\\ & \leq \|t_\alpha\| \cdot \|v - v_i\| + \epsilon + \| t_\infty \| \cdot \| v - v_i\| \leq 3 \epsilon
\end{align*}
Thus, $(t_\alpha)$ converges uniformly to $t_\infty$ on $K$. Therefore, pointwise convergence implies uniform convergence on compact subsets.
\end{proof}

\begin{proposition}\label{k5}
We have a bijection, natural in the compact Hausdorff space $S$, and the unital C*-algebras $A$ and $B$,
$$\mathrm{Hom} (A , C(S) \tensor B) \; \iso \; C( S, \mathrm{Hom}(A , B)),$$
where $\mathrm{Hom}(A , B)$ denotes the set of all unital $*$-homomorphisms from $A $ to $B$, equipped with the point-norm topology. It takes each each unital $*$-homomorphism $\phi \: A  \To C(S) \tensor B$ to the function $s \mapsto (\eval_s \tensor \mathrm{id}) (\phi( \, \cdot\,))$.
\end{proposition}

\begin{proof}
This proposition is just lemma \ref{k2}, with a different topology on $\Hom(A , B )$. The point-norm topology on $\Hom(A , B)$ is equal to the compact-open topology by lemma \ref{k3}. Thus, the topology on $\mathrm{Hom} (A,B)$ in lemma \ref{k2} is equivalently the kaonization of the point-norm topology. Since $S$ is a compact Hausdorff space, the set of continuous functions from $S$ to $\Hom(A , B)$ is the same if we equip $\Hom(A , B)$ with just the point-norm topology, without kaonizing it. Thus, the proposition follows.
\end{proof}

\begin{lemma}\label{k6}
Let $T$ be a compact Hausdorff space. 
The bijection of proposition \ref{k5} takes the canonical unital $*$-homomorphism $\pi\: C(T)  \To C(T) \tensor B$, defined by $\pi(a) = a \tensor 1$ to the continuous function $\imath\: T \To \mathrm{Hom}(C(T), B)$, defined by $\imath(t)(a) = a(t)\cdot 1$.
\end{lemma}

\begin{proof}
Let $p$ be the function from $T$ to $\mathrm{Hom}(C(T), B)$ that corresponds to $\pi$ under the bijection of proposition \ref{k5}.
For each $t \in T$, and all $a \in C(T)$, we compute:
$$p(t)(a) = (\mathrm{eval_t} \tensor \mathrm{id})(\pi(a)) = (\mathrm{eval_t} \tensor \mathrm{id}) (a \tensor 1) = a(t) \tensor 1 = a(t) \cdot 1.$$
\end{proof}

\begin{theorem}\label{k7}
Let $\tilde S$  and $T$ be compact Hausdorff spaces, and let $B$ be a unital $C^*$-algebra. Let $S$ be a closed subset of $\tilde S$, with inclusion function $j\: S \hookrightarrow \tilde S$. Let $f$ be a continuous function from $S$ to $T$. Use the notation of proposition \ref{k5} and lemma \ref{k6}. The following are equivalent:
\begin{enumerate}
\item The function $\imath \circ f \: S \To \mathrm{Hom}(C(T), B )$ extends to a continuous function $\tilde f\: \tilde S \To \mathrm{Hom}(C(T), B )$. $$
\begin{tikzcd}
\tilde S \arrow[dotted]{rrd}{\tilde f}
&
&
\\
S \arrow[hook]{u}{j}\arrow{r}[swap]{f}
&
T  \arrow{r}[swap]{\imath}
&
\mathrm{Hom}(C(T), B )
\end{tikzcd}
$$
\item There is a unital $*$-homomorphism $\phi\: C(T) \To C(\tilde S) \tensor B $ such that $$ (j^\star \tensor \mathrm{id}) \circ \phi = (f^\star \tensor \mathrm{id}) \circ \pi  .$$ $$
\begin{tikzcd}
C(\tilde S) \tensor B  \arrow{d}[swap]{j^\star \tensor \mathrm{id}}
&
&
\\
C(S) \tensor B 
&
C(T) \tensor B \arrow{l}{f^\star \tensor \mathrm{id}}
&
C(T) \arrow{l}{\pi} \arrow[dotted]{llu}[swap]{\phi} 
\end{tikzcd}
$$
\end{enumerate}
\end{theorem}

\begin{proof}
Apply proposition \ref{k5}, with $A=C(T)$ to obtain the following commutative diagram:
$$
\begin{tikzcd}
\Hom(C(T), C(\tilde S) \tensor B))
\arrow[leftrightarrow]{r}{\iso}
\arrow{d}[swap]{(j^\star \tensor \mathrm{id})\circ}
&
C(\tilde S, \Hom(C(T),B))
\arrow{d}{\circ j}
\\
\Hom(C(T), C(S) \tensor B))
\arrow[leftrightarrow]{r}{\iso} 
&
C(S, \Hom(C(T),B))
\\
\Hom(C(T), C(T) \tensor B))
\arrow[leftrightarrow]{r}{\iso}
\arrow{u}{(f^\star \tensor \mathrm{id})\circ}
&
C(T, \Hom(C(T),B))
\arrow{u}[swap]{\circ f}
\end{tikzcd}
$$
The homomorphism $\pi$ is an element of $\Hom(C(T), C(T) \tensor B))$, and the map $\imath$ is an element of $C(T, \Hom(C(T),B))$. By lemma \ref{k6}, the two elements correspond to each other under the indicated bijection. Condition (1) is equivalently that there exists an element $\tilde f$ of $C(\tilde S, \Hom(C(T),B))$ whose image in $C(S, \Hom(C(T),B))$ is equal to the image of $\imath$ in that set. Condition (2) is equivalently that there exists an element $\phi$ of $\Hom(C(T), C(\tilde S) \tensor B))$ whose image in $\Hom(C(T), C(S) \tensor B))$ us equal to the image of $\pi$ in that set. The commutative diagram demonstrates that the two conditions are equivalent, establishing the theorem.\end{proof}


We close this section by observing that the space $\Hom(C(T), M_n(\CC))$ is necessarily compact for every $n$.

\begin{lemma}\label{k9}
Let $A$ be a unital C*-algebra, and let $n$ be a positive integer. The space $\mathrm{Hom}(A , M_n(\CC))$ of unital $*$-homomorphisms from $A $ to $M_n(\CC)$, equipped with the compact-open topology, or equivalently, with the point-norm topology, is compact. If $A$ is separable, then $\mathrm{Hom}(A , M_n(\CC))$ has a countable basis for the point-norm topology.
\end{lemma}

\begin{proof}
The two topologies are equivalent as a special case of lemma \ref{k3}. To observe compactness, we put the set $\mathrm{Hom}(A , M_n(\CC))$ through the following isomorphisms in Steenrod's convenient category:
$$ C(A , M_n(\CC)) \; \iso \; C(A , \CC \times \cdots \times \CC) \; \iso \; C(A , \CC) \times \cdots \times C(A , \CC)$$
The set $\mathrm{Hom}(A , M_n(\CC))$ is closed in $C(A , M_n(\CC))$. A unital $*$-homomorphism has norm at most $1$, so $\mathrm{Hom}(A , M_n(\CC))$ corresponds to a closed subset of $L(A , \CC)_1 \times \cdots \times L(A , \CC)_1$. By lemma \ref{k3}, the compact-open topology on $L(A , \CC)_1$ is just the point-norm topology, or equivalently, the weak* topology. By Alaoglu's theorem, the unit ball $L(A , \CC)_1$ is compact. Thus, the topology on $L(A , \CC)_1$ as a subspace of $C(A , \CC)$ in Steenrod's convenient category is just the weak* topology. Therefore the space $\mathrm{Hom}(A , M_n(\CC))$ is compact. Furthermore, if $A$ is separable, then the unit ball $L(A , \CC)_1$ is metrizable \cite{Pedersen}*{exercise 2.5.3}, implying the same for $\mathrm{Hom}(A , M_n(\CC))$. Of course, for compact Hausdorff spaces metrizability is equivalent to the existence of a countable basis.
\end{proof}

\section{quantum loops}

We write $S^1$ for the unit circle, $D^2$ for the unit disk, and $j\: S^1 \hookrightarrow D^2$ for the inclusion map.

\begin{definition}\label{l1}
Let $T$ be a compact Hausdorff space. A loop in $T$ is a function $f\: S^1 \To T$. A loop $f$ is \underline{quantum nullhomotopic} just in case there is a nonzero finite-dimensional Hilbert space $H$ and a unital $*$-homomorphism $\phi\:C(T) \To C(D^2) \tensor L(H)$ such that $ (j^\star \tensor \mathrm{id})\phi(a) = f^\star(a) \tensor 1_H$ for all $a \in C(T)$.
$$
\begin{tikzcd}
C(D^2) \otimes L(H) \arrow{d}[swap]{j^\star \tensor \mathrm{id}}
&
&
\\
C(S^1) \otimes L(H)
&
C(T) \otimes L(H) \arrow{l}{f^\star \tensor \mathrm{id}}
&
C(T) \arrow[dotted]{llu}[swap]{\phi} \arrow{l}{a \tensor 1 \mapsfrom a}
\end{tikzcd}
$$
\noindent Equivalently, $f$ is nullhomotopic if and only if there exists an positive integer $n$, and a unital $*$-homomorphism $\phi\: C(T) \To M_n(C(D^2))$ such that applying $j^\star$ to every element of the matrix $\phi(a)$ yields the matrix $f^\star(a) \cdot I_n$ for all $a \in C(T)$.
\end{definition}

\begin{corollary}\label{l2}
Let $T$ be a compact Hausdorff space. A loop $f$ in $T$ is quantum nullhomotopic if and only if there is a positive integer $n$ such that the loop $\imath_\star(f) = \imath \circ f$ in $\Hom(C(T), M_n(\CC))$ is nullhomotopic in the ordinary sense, where $\imath$ is defined by $\imath(t)(a) = a (t) \cdot I_n$.
\end{corollary}

\begin{proof}
Apply theorem \ref{k7}, with $S = S^1$ and $\tilde S = D^2$.
\end{proof}

\begin{proposition}\label{l3}
Let $g$ be a continuous function from a compact Hausdorff space $T_1$ to a compact Hausdorff space $T_2$. If a loop $f\: S^1 \To T_1$ is quantum nullhomotopic in $T$, then the loop $g \circ f \: S^1 \To T_2$ is quantum nullhomotopic in $T_2$.
\end{proposition}

\begin{proof} Apply corollary \ref{l2}.  Let $n$ be a positive integer such that the loop $\imath_\star(f)\: S^1 \To \Hom(C(T_1), M_n(\CC))$ is nullhomotopic in the ordinary sense. We show that the loop $ \imath_\star (g \circ f)\: S^1 \To \Hom(C(T_2), M_n(\CC) )$ is a continuous image of the loop $\imath_\star(f)$. 

$$
\begin{tikzcd}
S^1 \arrow{r}{\imath_\star(f)} \arrow{rd}[swap]{\imath_\star (g \circ f)}
&
\Hom(C(T_1), M_n(\CC)) \arrow{d}{\circ g^\star}
\\
&
\Hom(C(T_2), M_n(\CC))
\end{tikzcd}
$$

For each point $s$ on the circle $S^1$, and all $a_2 \in C(T_2)$, we calculate that
$$\imath_\star (g \circ f)(s)(a_2) = a_2(g(f(s))) \cdot I_n = g^\star(a_2)(f(s)) \cdot I_n = \imath_\star(f)(s) ( g^\star(a_2)) = (\imath_\star(f)(s) \circ g^\star)(a_2).$$
Thus, the homomorphism $\imath_\star (g \circ f)(s)$ is obtained by composing $\imath_\star(f)(s)$ with $g^\star$. This is a continuous function from $\Hom(C(T_1), M_n(\CC) )$ to $\Hom(C(T_2), M_n(\CC) )$, since $\Hom(C(\, - \,), M_n(\CC) )$ is a functor from the category of compact Hausdorff spaces and continuous functions to Steenrod's convenient category. Being a continuous image of a nullhomotopic loop, the loop $\imath_\star (g \circ f)$ is also nullhomotopic. Appealing a second time to corollary \ref{l2}, we conclude that $g \circ f$ is quantum nullhomotopic.
\end{proof}

\begin{proposition}\label{l4}
The identity function $f\:S^1 \To S^1$ is not quantum nullhomotopic as a loop in $S^1$.
\end{proposition}

\begin{proof}
Let $f \:S^1 \To S^1$ be the identity function, and suppose that $f$ is quantum nullhomotopic as a loop in $S^1$. By corollary \ref{l2}, there is a positive integer $n$ such that the loop $\imath_\star(f)\: S^1 \To \Hom (C(S^1), M_n(\CC))$ is nullhomotopic. We will compose $\imath_\star(f)$ with two other continuous functions to obtain a loop which we know is not nullhomotopic, arriving at a contradiction.

First, let $z$ be the inclusion of $S^1$ into $\CC$; it is a unitary operator in $C(S^1)$. Evaluation at $z$ is a function from $\Hom(C(S^1), M_n(\CC))$ to the space $U(n)$ of unitary $n \times n$ matrices. This function $\eval_z$ is continuous because all evaluation functions are continuous in Steenrod's convenient category.

Second, the determinant is a continuous function from $U(n)$ to $S^1$. Thus, we obtain a nullhomotopic loop $\mathrm{det} \circ \eval_z \circ f$ in $S^1$.
$$
\begin{tikzcd}
S^1 \arrow{r}{\imath_\star(f)}
&
\Hom(C(S^1, M_n(\CC)) \arrow{r}{\eval_z}
&
U(n) \arrow{r}{\mathrm{det}}
&
S^1
\end{tikzcd}
$$
For each point $s$ on the circle $S^1$ we calculate that
$$ ( \mathrm{det} \circ \eval_z \circ \imath_\star(f)) (s) = \mathrm{det}(z(f(s)) \cdot I_n) = \mathrm{det}(s \cdot I_n) = s^n.$$
It is a basic fact that this loop is not nullhomotopic. We have reached a contradiction. Therefore, $f$ is not quantum nullhomotopic.
\end{proof}

\begin{lemma}\label{l5}
Let $B$ a be a unital C*-algebra, and let $z\: S^1 \hookrightarrow \CC$ be the inclusion map. Evaluation at $z$ is a homeomorphism 
$$\Hom(C(S^1), B) \To U(B),$$
where $U(B)$ denotes the set of unitary operators in $B$, equipped with the norm topology.
\end{lemma}

\begin{proof}
Recall that the C*-algebra $C(S^1)$ is isomorphic to the group C*-algebra $\mathrm C^*(\ZZ)$ \cite{Davidson}*{proposition VII.1.1}. The inclusion function $z \: S^1 \hookrightarrow \CC$ is mapped to the generator $u_1$ of $\mathrm C^*(\ZZ)$. The C*-algebra $\mathrm C^*(\ZZ)$ is evidently the universal C*-algebra for a single unitary operator: for any C*-algebra $B$, there is a bijective correspondence between the unital $*$-homomorphisms from $\mathrm C^*(\ZZ)$ to $B$, and the unitary operators in $B$. Each such unital $*$-homomorphism sends the generator $u_1$ to the corresponding unitary in $B$. Therefore, we have a bijection $\eval_{z}$ from $\Hom(C(S^1), B)$ to $U(B)$, defined by $\rho \mapsto \rho(z)$.

The bijection $\eval_{z}$ is continuous, because all evaluation functions are continuous in Steenrod's convenient category. To show that it is a homeomorphism, let $(\rho_\lambda)$ be a net in $\Hom(C(S^1), B)$ whose image $(\eval_{z}(\rho_\lambda))$ converges to $\eval_{z}(\rho)$ for some homomorphism $\rho$ in $\Hom(C(S^1), B)$. In other words, $(\rho_\lambda(z))$ converges to $\rho(z)$. The operations that make up the $*$-algebra structure of $B$ are all continuous in the norm topology, so in fact, $(\rho_\lambda(a))$ converges to $\rho(a)$ for all $a$ in the unital $*$-subalgebra $A_0$ of $C(S^1)$ that is generated by $z$.

Let $a$ be any element of $C(S^1)$. The algebra $A_0$ is dense in $C(S^1)$ by the Stone-Weierstrass theorem, so we can choose a sequence $(a_n)$ in $A_0$ such that $\|a_n - a \| \leq 1/n$ for each positive integer $n$. We now estimate that for each positive integer $n$,
\begin{align*}
\|\rho(a) - \rho_\lambda(a) \| &\leq \|\rho(a) - \rho(a_n) \| + \| \rho(a_n) - \rho_\lambda(a_n)\| + \| \rho_\lambda(a_n) - \rho_\lambda(a)\|
\\ &
\leq \| \rho\| \cdot \| a_n - a\| + \| \rho(a_n) - \rho_\lambda(a_n)\| + \|\rho_\lambda\| \cdot \| a_n - a\|
\\ &
\leq \frac 1 n + \| \rho(a_n) - \rho_\lambda(a_n)\| + \frac 1 n 
\\ &= \frac 2 n + \| \rho(a_n) - \rho_\lambda(a_n)\|
\end{align*}
The term $\| \rho(a_n) - \rho_\lambda(a_n)\|$ converges to $0$ as $\lambda$ goes to infinity, so $\limsup_\lambda \|\rho(a) - \rho_\lambda(a) \| \leq 2 /n$. We have this estimate for each positive integer $n$, so $\lim_\lambda \| \rho (a) - \rho_\lambda(a)\| = 0$. In other words, $\rho_\lambda$ converges to $\rho$ in the point-norm topology. Therefore, $\eval_{z}^{-1} \: U(B) \To \Hom(C(S^1), B)$ is norm-(point-norm) continuous, and thus, by lemma \ref{k3}, it is norm-(compact-open) continuous.

It is an elementary fact about compactly generated spaces that kaonization does not introduce or eliminate compact subset. In this case, if a set $K \subsetof \Hom(C(S^1), B)$ if compact in the kaonized compact-open topology, then it is compact in the compact-open topology itself, and vice versa. Thus, we have shown that $\eval_{z}$ induces a one-to-one correspondence between the compact subsets of $\Hom(C(S^1), B)$ and the compact subsets of $U(B)$.  Since both spaces are compactly generated, $\eval_z$ must be a homeomorphism.
\end{proof}

\begin{theorem}
Let $H$ be a separable infinite-dimensional Hilbert space.
There is a unital $*$-homomorphism $\phi\: C(S^1) \To C(D^2) \tensor L(H)$ such that $(j^\star \tensor \mathrm{id}) (\phi(a)) = a \tensor 1_H$ for all $a \in C(S^1)$, where $j \: S^1 \hookrightarrow D^2$ is the inclusion map.
$$
\begin{tikzcd}
C(D^2) \tensor L(H)
\arrow{d}[swap]{j^\star \tensor \mathrm{id}}
&
&
\\
C(S^1) \tensor L(H)
&
&
C(S^1)
\arrow{ll}{ a \tensor 1 \mapsfrom a}
\arrow{ull}[swap]{\phi}
\end{tikzcd}
$$
\end{theorem}

The proof of this theorem is a more elaborate variation on the proof of proposition \ref{l4}. We write $U(H)$ for the set of unitary operators in $L(H)$, equipped with the norm topology.

\begin{proof}
We apply theorem \ref{k7}, with $S = S^1$, $T = S^1$, $\tilde S = D^2$, $B = L(H)$, and $f = \mathrm{id}_{S^1}$. Specifically, we show that the function $\imath\: S^1 \To \Hom(C(S^1), L(H))$, defined by $\imath(s)(a) = a(s) \cdot 1_H$ for all $a \in C(S^1)$ and $s \in S^1$, extends to $D^2$.

Composing $\imath$ with the homeomorphism $\eval_z$ of lemma \ref{l5}, we obtain the loop $\eval_z \circ \imath$ in $U(H)$. By Kuiper's theorem \cite{Kuiper}*{theorem (3)}, $U(H)$ is contractible, so the loop $\eval_z \circ \imath$ is nullhomotopic. Since $\eval_z$ is a homeomorphism by lemma \ref{l5}, the loop $\imath$ must also be nullhomotopic. Applying, theorem \ref{k7}, we reach the desired conclusion.
\end{proof}

\section{every loop in $\RP ^2$ is quantum nullhomotopic}

We define the circle $S^1$ to be a subspace of $\CC$, and the sphere $S^2$ to be a subspace of $\CC \times \RR$. Explicitly, the circle $S^1$ consists of complex numbers $\alpha$ such that $|\alpha|^2 = 1$, and the sphere $S^2$ consists of pairs $(\alpha, t)$ in $\CC \times \RR$ such that $|\alpha|^2 +t^2 =1$. The sphere $S^2$ is homeomorphic to $\CP^1$; we explicitly describe one such homeomorphism that is convenient in this context.

Write $R_1(M_2(\CC))$ for the space of reflection matrices of negative determinant, i. e., for the space of $2 \times 2$ complex matrices $b$ such that $b^* = b$, $b^2 = 1$, and $\mathrm{det}(b) = -1$. There is a nice homeomorphism $h\: S^2 \To R_1(M_2(\CC))$ given by the formula
$$ h(\alpha, t) = \left( \begin{matrix} t & \overline \alpha \\ \alpha & -t \end{matrix} \right).$$ 
It satisfies $h(-x) = - h(x)$ for all points $x \in S^2$.

The elements of $R_1(M_2(\CC))$ are exactly the matrices with spectrum equal to $\{1, -1\}$. Thus, for any matrix $b$ in $R_1(M_2(\CC))$, the matrices $ (1+b)/2$ and $(1-b)/2$ are rank-one projections, orthogonal to each other. We obtain a pair of homeomorphisms $x \mapsto (1 + h(x))/2$, and $x \mapsto (1-h(x))/2$, from the space $S^2$ to the space $P_1(M_2(\CC))$ of rank-one projections, with the property that their values at every point $x$ of the sphere $S^2$ are orthogonal.

\begin{lemma}\label{m1}
Let $T$ be a compact Hausdorff space. The space $\Hom(C(T), M_2(\CC))$ is a quotient of $S^2 \times T \times T$:
$$ 
\begin{tikzcd}
S^2\times T \times T
\arrow[two heads]{r}{q_1}
&
(S^2 \times T \times T) / \ZZ_2
\arrow[two heads]{r}{q_2}
&
\Hom(C(T), M_2(\CC))
\end{tikzcd}
$$
All three spaces are compact Hausdorff spaces. The space $(S^2 \times T \times T) / \ZZ_2$ is the quotient of $S^2 \times T \times T$ by the involution $(x, t_1, t_2) \mapsto (-x, t_2, t_1)$, and $q_1$ is the corresponding quotient map. The quotient map $q_2$ takes each orbit $\{(x, t_1, t_2), (-x, t_2, t_1)\}$ to the homomorphism 
\begin{align}\tag{$\ast$}
\mathrm{eval}_{t_1}(\,\cdot\,) \frac{I_2+ h(x)}2  + \mathrm{eval}_{t_2}(\, \cdot\,) \frac{I_2 - h(x)}2.
\end{align}
The quotient map $q_2$ identifies distinct orbits $\{(x, t_1, t_2), (-x, t_2, t_1)\}$ and $\{(x', t_1', t_2'), (-x', t_2', t_1')\}$ if and only if $t_1 = t_2 = t_1' = t_2'$.
\end{lemma}

\begin{proof}
The first space $S^2\times T \times T$ is a compact Hausdorff space because it is a product of compact Hausdorff spaces. The second space $(S^2 \times T \times T) / \ZZ_2$ is a compact Hausdorff space because it is the quotient of a compact Hausdorff space by the action of a finite group. The third space $\Hom(C(T), M_2(\CC))$ is a compact Hausdorff space by lemma \ref{k9}.

The expression ($\ast$) clearly names a morphism in Steenrod's convenient category, i. e., a continuous function from $S^2 \times T \times T$ to $C(C(T, \CC), M_2(\CC))$. It is immediately apparent from the expression that this continuous function is invariant under the action of $\ZZ_2$. Therefore, it factors through the quotient map $q_1$, via a continuous function $q_2$.

For all points $x$ in $S^2$, the matrices $(1 + h(x))/2$ and $(1-h(x))/2$ are rank-one projections that sum to the identity matrix. Furthermore, for all points $t_1$ and $t_2$ in $T$ the evaluation functions $\eval_{t_1}$ and $\eval_{t_2}$ are unital $*$-homomorphisms. It follows that for all points $(x, t_1, t_2)$ in the product space $S^2 \times T \times T$, the expression ($\ast$) names a unital $*$-homomorphism. Therefore, the range of $q_2$ is a subset of $\Hom(C(T), M_2(\CC))$.

We collect a couple of basic facts. First, the range of any element $\rho$ in $\Hom(C(T), M_2(\CC))$ is a commutative unital C*-subalgebra of $M_2(\CC)$. In particular, it is isomorphic to either $\CC$ or $\CC^2$. Second, if $\rho= (q_2 \circ q_1)(x, t_1, t_2)$ for some $(x, t_1, t_2) \in S^2 \times T \times T$, then the dimension of $\rho(C(T))$ depends on whether the points $t_1$ and $t_2$ are equal or distinct. If they are equal, then $\rho = \eval_{t_1}( \cdot) I_2$, so $\rho(C(T)) \iso \CC$. If $t_1$ and $t_2$ are distinct, then the distinct characters $\eval_{t_1}$ and $\eval_{t_2}$ both factor through $\rho\: C(T) \To \rho(C(T))$, so $\rho(C(T))$ cannot be one-dimensional, and thus, $\rho(C(T)) \iso \CC^2$.

Let $\rho$ be any element of $\Hom(C(T), M_2(\CC))$. First, assume that $\rho(C(T)) \iso \CC$. It follows that $\rho$ is the composition of a unital $*$-homomorphism $\gamma\: C(T) \To \CC$ with the canonical inclusion $ \CC \hookrightarrow M_2(\CC)$. By Gelfand duality, any such homomorphism $\gamma$ is equal to $\eval_t$ for some point $t$. Thus, $\rho = \eval_t(\, \cdot \,) I_2 = q_2 (q_1((1,0), t, t))$.

Now assume that $\rho(C(T)) \iso \CC^2$. It follows that we can write $\rho$ as $\gamma_1(\,\cdot\,) b_1 + \gamma_2 (\, \cdot\,) b_2$, for pairwise orthogonal projections $b_1$ and $b_2$ in $\rho(C(T))$, and unital $*$-homomorphisms $\gamma_1$ and $\gamma_2$ from $C(T)$ to $\CC$. By Gelfand duality, there exist points $t_1$ and $t_2$ such that $\gamma_1 = \eval_{t_1}$ and $\gamma_2 = \eval_{t_2}$. There also exists a point $x$ in the sphere $S^2$ such that $b_1 = (1+h(x))/2$, and consequently, such that $b_2 = (1-h(x))/2$. Thus, $\rho = \eval_{t_1}(\, \cdot\,)((1+h(x))/2) + \eval_{t_2}(\, \cdot\,)((1-h(x))/2) = q_2(q_1(x, t_1, t_2))$. Therefore, $q_2$ is surjective. As any continuous surjective function between compact Hausdorff spaces it must be a quotient map.

Let $w = \{(x, t_1, t_2), (-x, t_2, t_1)\}$ and $w' =\{(x', t_1', t_2'), (-x', t_2', t_1')\}$ be orbits in $(S^2 \times T \times T) / \ZZ_2$. If $t_1 = t_2 = t_1' = t_2'$, then $q(w) = \eval_{t_1}(\,\cdot\,) I_2 = \eval_{t_1'}(\, \cdot\,)I_2 = q_2(w')$. Conversely, assume that $q_2(w) = q_2(w')$, and write $\rho = q_2(w) = q_2(w')$. The range $\rho(C(T))$ is isomorphic to either $\CC$ or $\CC^2$. If it is isomorphic to $\CC$, then $t_1 = t_2$, $t_1' = t_2'$, and $\eval_{t_1}(\, \cdot \,)I_2 = q_2(w) = q_2(w') = \eval_{t_1'}(\, \cdot \, ) I_2$, so $t_1 =t_1'$. Thus, $t_1 = t_2 = t_1' = t_2'$.

Therefore, assume that the range $\rho(C(T))$ is isomorphic to $\CC^2$. The two minimal projections in $q_2(w)(C(T))$ are $(1+h(x))/2$ or $(1-h(x))/2$, and similarly the two minimal projections in $q_2(w')(C(T))$ are $(1+h(x'))/2$ or $(1-h(x'))/2$. Since the C*-algebras $q_2(w)(C(T))$ and $q_2(w')(C(T))$ are equal, the projection $(1+h(x))/2$ must be equal to either $(1+h(x'))/2$ or to $(1-h(x'))/2$. In other words, $h(x)$ must be equal to either $h(x')$ or $h(-x')$, so $w = w'$. Therefore, 
$q_2(w) = q_2(w')$ if and only if $w = w'$ or $t_1 = t_2 = t_1' = t_2'$.
\end{proof}

If $\psi$ is a path in $S^2$, and $\phi_1$ and $\phi_2$ are paths in $T$, we write $[\psi, \phi_1, \phi_2]$ for the homotopy class of the path $(\psi, \phi_1, \phi_2)$ defined by $\tau \mapsto (\psi(\tau), \phi_1(\tau), \phi_2(\tau))$.

\begin{theorem}\label{m2}
Let $T$ be a compact Hausdorff space with distinguished point $t_0$. Write $q$ for the composition of the quotient maps $q_1$ and $q_2$ in lemma \ref{m1}, and write 
$$q_*\: \pi_1(S^2 \times T \times T, ((1,0), t_0, t_0)) \To \pi_1(\Hom(C(T), M_2(\CC)), \eval_{t_0})$$
for the induced homomorphism of fundamental groups. For every path $\varphi\: [0,1] \To T$, beginning and ending at $t_0$, we have that $q_*[\cst_{(1,0)}, \varphi, \cst_{t_0}] = q_*[\cst_{(1,0)}, \cst_{t_0}, \varphi]$.
\end{theorem}

\begin{proof}
Let $\sigma$ be the path in $S^2$, beginning at $(1,0)$ and ending at $(-1,0)$, defined by $\sigma(\tau) = (\exp(\pi i \tau), 0)$. The path $q \circ (\sigma, \cst_{t_0}, \cst_{t_0})$ is the constant path at $\eval_{t_0}$, by lemma \ref{m1}. It follows that $q_*[\sigma, \cst_{t_0}, \cst_{t_0}]$ is the identity element of $\pi_1(\Hom(C(T), M_2(\CC)), \eval_{t_0})$.

Let $\varphi\:[0,1] \To T$ be any path in $T$ beginning and ending at $t_0$. Appealing to lemma \ref{m1} in the first step, we calculate that
\begin{align*}
q_*[\cst_{(1,0)}, \varphi, \cst_{t_0}]
&= q_*[\cst_{(-1,0)}, \cst_{t_0}, \varphi]
\\ &= q_*[\cst_{(-1,0)}, \cst_{t_0}, \varphi] \cdot q_*[\sigma, \cst_{t_0}, \cst_{t_0}]
\\ &= q_*( [ \cst_{(-1,0)}, \cst_{t_0}, \varphi] \cdot [\sigma, \cst_{t_0}, \cst_{t_0}])
\\ &= q_*[\sigma, \cst_{t_0}, \varphi]
\\ &= q_*([\sigma, \cst_{t_0}, \cst_{t_0}] \cdot [ \cst_{(1,0)}, \cst_{t_0}, \varphi])
\\ &= q_*[\sigma, \cst_{t_0}, \cst_{t_0}] \cdot q_*[ \cst_{(1,0)}, \cst_{t_0}, \varphi]
\\ &= q_*[ \cst_{(1,0)}, \cst_{t_0}, \varphi]
\end{align*}
\end{proof}

The embedding $\imath \: T \rightarrowtail \Hom(C(T), M_2(\CC))$ clearly factors through $q$:
$$
\begin{tikzcd}
&
S^2 \times T \times T
\arrow{d}{q}
\\
T
\arrow{ru}{\cst_{(1,0)} \times \Delta}
\arrow{r}[swap]{\imath}
&
\Hom(C(T), M_2(\CC))
\end{tikzcd}
$$
The map diagonal map $\Delta\: T \To T \times T$ is defined by $\Delta(t) = (t, t)$.

\begin{corollary}
Every loop in $\RP^2$ is quantum nullhomotopic.
\end{corollary}

\begin{proof}
Let $t_0$ be any point of $T = \RP^2$, and let $\varphi\: [0,1] \To T$ be any path that begins and ends at $t_0$.
\begin{align*}
 \imath_*[\varphi] & = q_*( (\cst_{(1,0)} \times \Delta)_* [\varphi]) = q_*[ \cst_{(1,0)}, \varphi, \varphi]
 \\ &= q_*([\cst_{(1,0)}, \varphi, \cst_{t_0}] \cdot [\cst_{(1,0)}, \cst_{t_0},\varphi ])
  \\ &= q_*[\cst_{(1,0)}, \varphi, \cst_{t_0}] \cdot q_* [\cst_{(1,0)}, \cst_{t_0},\varphi]
  \\ & = q_*[\cst_{(1,0)}, \varphi, \cst_{t_0}] \cdot q_* [\cst_{(1,0)},\varphi, \cst_{t_0}]
  \\ & = q_*([\cst_{(1,0)}, \varphi, \cst_{t_0}] \cdot [\cst_{(1,0)},\varphi, \cst_{t_0}])
\end{align*}
\end{proof}

Thus, $\imath_*[\phi]$ is the identity element of $\pi_1(\Hom(C(T), M_2(\CC)), \eval_{t_0})$. Therefore, for every loop $f\: S^1 \To T$, the loop $\iota \circ f$ is nullhomotopic. By corollary \ref{l2}, we conclude that every loop in $T = \RP^2$ is quantum nullhomotopic.

\begin{proposition}\label{m4}
Let $T$ be a compact Hausdorff space with distinguished point $t_0$. The group $\imath_*(\pi_1(T, t_0))$ is commutative.
\end{proposition}

\begin{proof}
Let $\imath_*[\varphi_1]$ and $\imath_*[\varphi_2]$ be elements of $\imath_*(\pi_1(T, t_0))$.  Appealing to theorem \ref{m2}, we calculate that
\begin{align*}
  q_*[ \cst_{(1,0)},  \varphi_1, \cst_{t_0} ] \cdot q_*[ \cst_{(1,0)},  \varphi_2, \cst_{t_0}] 
& =
  q_*[ \cst_{(1,0)},  \varphi_1, \cst_{t_0}] \cdot q_*[ \cst_{(1,0)},  \cst_{t_0}, \varphi_2]
\\ &=
q_*[\cst_{(1,0)},  \varphi_1, \varphi_2 ]
\\ &= q_*[ \cst_{(1,0)},  \cst_{t_0}, \varphi_2] \cdot  q_*[ \cst_{(1,0)},  \varphi_1, \cst_{t_0} ]
\\ & = q_*[ \cst_{(1,0)},  \varphi_2, \cst_{t_0}]  \cdot q_*[ \cst_{(1,0)},  \varphi_1, \cst_{t_0} ]
\end{align*}
Similarly $q_*[ \cst_{(1,0)},  \cst_{t_0},\varphi_1  ]$ and $q_*[ \cst_{(1,0)},   \cst_{t_0},\varphi_2]$ commute. Since we can write $$\imath_*[\varphi_1] = q_*[ \cst_{(1,0)}, \cst_{t_0}, \varphi_1] \cdot q_*[ \cst_{(1,0)},  \varphi_1, \cst_{t_0} ],$$
and likewise for $\imath_*[\varphi_2]$,
it follows that $\imath_*[\varphi_1]$ and $\imath_*[\varphi_2]$ commute. Therefore, the group $\imath_*(\pi_1(T, t_0))$ is commutative.
\end{proof}

\begin{corollary}
There is a loop in $(S^1, 1) \vee (S^1, 1)$ that is quantum nullhomotopic, but not nullhomotopic in the ordinary sense.
\end{corollary}

\begin{proof}
The fundamental group of $T =S^1 \vee S^1$ is the free group on two generators. Therefore, the commutator of its two generators yields a loop $f$ in $S^1 \vee S^1$ that is not nullhomotopic. However, the image of this loop $f$ in $\Hom(C(T), M_2(\CC))$ must be nullhomotopic by proposition \ref{m4}. It follows that $f$ is quantum nullhomotopic by corollary \ref{l2}.
\end{proof}

\end{document}